\documentclass[11pt]{article}
\usepackage[T1]{fontenc}
\usepackage{lmodern}
\usepackage{microtype}
\usepackage{geometry}
\usepackage{amsmath,amssymb,amsthm,mathtools}
\usepackage[shortlabels]{enumitem}
\usepackage{xcolor}
\usepackage[hidelinks]{hyperref}
\usepackage[nameinlink,noabbrev]{cleveref}

\allowdisplaybreaks

\newtheorem{theorem}{Theorem}[section]
\newtheorem{proposition}[theorem]{Proposition}
\newtheorem{lemma}[theorem]{Lemma}
\newtheorem{corollary}[theorem]{Corollary}
\theoremstyle{definition}

\theoremstyle{remark}
\newtheorem{remark}[theorem]{Remark}

\title{A Lang-Trotter Problem for Non-Geometric Quadratic Inductions}
\author{Haoyang Yuan\\Department of Mathematics, Nanjing University, Nanjing 210093, China\\zgqyhy@163.com}

\date{}

\begin{document}

\maketitle

\begin{abstract}
Let $K/\mathbb Q$ be an imaginary quadratic extension and $p$  an odd prime. Write $\rho=\operatorname{Ind}_{G_K}^{G_{\mathbb Q}}\chi$, where $E/\mathbb Q_p$ is a finite extension and  $\chi:G_K\to\mathcal O_E^\times$ is a continuous character. For a fixed $r\in\mathbb Z\setminus\{0\}$, let $\pi_{\rho,r}(X)$ denote the number of rational primes $\ell\le X$ such that $\rho$ is unramified at $\ell$ and $\operatorname{tr}\rho(\operatorname{Frob}_\ell)=r$. Let $a,b$ be the two weights of $\chi$ at $p$. We prove that if $(a,b)\notin\mathbb Q^2$, then $\pi_{\rho,r}(X)\ll_{\rho,r,\varepsilon}X^\varepsilon$ for every $\varepsilon>0$, while if $(a,b)\in\mathbb Q^2\setminus\mathbb Z^2$, then only finitely many such primes occur. These bounds are substantially sparser than the classical CM Lang--Trotter scale. The main input in the non-rational case is a rigidity theorem for algebraic curves in the \(p\)-adic analytic trace locus, combined with rigid-analytic Pila--Wilkie counting; the rational non-integral case is treated by a local ramification argument.
\end{abstract}

\textbf{2020 Mathematics Subject Classification.}
Primary 11F80; Secondary 11R37, 11R44, 14G22.

\textbf{Keywords:}
Lang--Trotter conjecture; fixed Frobenius traces; $p$-adic Galois representations; quadratic induction; rigid analytic geometry.

\section{Introduction}

Let $L/\mathbb Q_p$ be a finite extension, and let
\[
\rho:G_{\mathbb Q}\longrightarrow \mathrm{GL}_2(L)
\]
be a continuous $p$-adic representation. It is geometric in the sense of Fontaine-Mazur \cite{FontaineMazur1995} if it is unramified outside finitely many primes and de Rham at $p$. For $r\in L$ and $X\ge2$, put
\begin{equation}\label{eq:counting-function}
\pi_{\rho,r}(X)
=
\#\left\{
\ell\le X:
\rho\text{ is unramified at }\ell,
\ \operatorname{tr}\rho(\operatorname{Frob}_\ell)=r
\right\},
\end{equation}
where $\ell$ is a rational prime and $\operatorname{Frob}_\ell$ is arithmetic Frobenius.  For an elliptic curve $E/\mathbb Q$ of good reduction at $\ell$, write
\[
a_\ell(E)=\ell+1-\#E(\mathbb F_\ell).
\]
For an elliptic curve $E/\mathbb Q$ and a fixed integer $r$(with $r\neq 0$ if $E$ has CM), the Lang-Trotter conjecture predicts, when the corresponding constant $C_{E,r}$ is nonzero,
\begin{equation}\label{eq:lang-trotter}
\#\{\ell\le X:a_\ell(E)=r\}
\sim
C_{E,r}\frac{\sqrt X}{\log X}.
\end{equation}
Since $\pi(X)\sim X/\log X$, the conjecture predicts that primes with a fixed Frobenius trace have density zero, but occur on the characteristic scale $\sqrt X/\log X$.  This is the classical Lang-Trotter scale with which we compare our estimates below.  The conjecture originates in \cite{LangTrotter1976}; for a broad overview, see \cite{Katz2009}, and for representative upper-bound results, see \cite{Zywina2015,ThornerZaman2018}.

For CM elliptic curves, fixed Frobenius traces are naturally related to Hecke characters. Recent work on the connection between the Lang-Trotter conjecture and the Hardy-Littlewood conjecture includes the work of Qin \cite{Qin2016,Qin2025} on elliptic curves $y^2=x^3+D$ and $y^2=x^3+Dx$  and the work of Hu-Lei-Qin \cite{HuLeiQin} on other CM elliptic curves. For similar results and unconditional bounds, see Wan-Xi \cite{WanXi2026}.   

Pande \cite{pande2011} constructed infinitely ramified
two-dimensional $p$-adic Galois representations whose Frobenius-trace
distributions violate Lang-Trotter-type predictions.  By contrast,
the representations considered here are unramified outside finitely
many primes; their atypical fixed-trace behavior comes from the
non-de Rham weights at $p$.

We instead consider quadratic inductions of arbitrary continuous $p$-adic characters, without assuming local algebraicity at $p$.  The resulting representations need not be geometric. We study the fixed-trace problem for these non-geometric quadratic inductions.

More precisely, let $K/\mathbb Q$ be imaginary quadratic, let $p$ be
odd, let $E/\mathbb Q_p$ be a finite extension, and let
\[
\chi:G_K\longrightarrow\mathcal O_E^\times
\]
be continuous. Set
\[
\rho=\operatorname{Ind}_{G_K}^{G_{\mathbb Q}}\chi.
\]
After a finite extension of coefficients, let $a,b$ denote the two weights of $\chi$ at $p$, as defined in Section~\ref{sec:local}.

\begin{theorem}\label{thm:main}
Let $K/\mathbb Q$ be imaginary quadratic, $p$ an odd prime and $E/\mathbb Q_p$ a finite extension. Let 
\[
\chi:G_K\longrightarrow\mathcal O_E^\times
\]
be a continuous character.
Put
\[
\rho=\operatorname{Ind}_{G_K}^{G_{\mathbb Q}}\chi,
\]
and let $a,b$ be the weights of $\chi$ at $p$. Fix $r\in\mathbb Z\setminus\{0\}$.
\begin{enumerate}[(i)]
\item If $(a,b)\notin\mathbb Q^2$, then, for every $\varepsilon>0$,
\[
\pi_{\rho,r}(X)\ll_{\rho,r,\varepsilon}X^\varepsilon.
\]
\item If $(a,b)\in\mathbb Q^2\setminus\mathbb Z^2$, then
\[
\pi_{\rho,r}(X)=O_{\rho,r}(1).
\]
\end{enumerate}
\end{theorem}

\begin{remark}[The case $r=0$]\label{rem:r-zero}
The restriction $r\ne0$ is essential.  If $\ell$ is unramified for $\rho$ and inert in $K$, then the character of the index-two induced representation vanishes on the non-trivial coset of $G_K$, and hence
\[
\operatorname{tr}\rho(\operatorname{Frob}_\ell)=0.
\]
After removing the fixed finite ramification set, the inert primes have density $1/2$ by the Chebotarev density theorem; see, for example, \cite[Chapter~VII]{Neukirch1999}.  Consequently
\[
\pi_{\rho,0}(X)\gg_{K,\rho}\frac{X}{\log X},
\]
so neither conclusion of \cref{thm:main} can hold for $r=0$ in general.
\end{remark}

All implied constants may depend on the fixed representation, $r$, and $\varepsilon$; finite exceptional sets are absorbed into $O_{\rho,r}(1)$.  In the diagonal and anti-diagonal non-rational cases, \cref{cor:diagonal-antidiagonal} strengthens the first conclusion to finiteness.

The contrast with the classical Lang-Trotter conjecture is immediate.

\begin{corollary}\label{cor:nongeometric}
Under the hypotheses of \cref{thm:main}, the representation $\rho$ is
unramified outside finitely many primes. Moreover,
\[
\rho\text{ is de Rham at }p
\quad\Longleftrightarrow\quad
a,b\in\mathbb Z.
\]
Consequently, $\rho$ is geometric in the sense of Fontaine-Mazur if
and only if $a,b\in\mathbb Z$. Therefore, if $\rho$ is non-geometric,
then, for every $\varepsilon>0$,
\[
\pi_{\rho,r}(X)\ll_{\rho,r,\varepsilon}X^\varepsilon.
\]
In particular,
\[
\pi_{\rho,r}(X)
=o\!\left(\frac{\sqrt X}{\log X}\right).
\]
\end{corollary}

For non-rational weights, we combine a rigidity result for algebraic curves in the trace locus with $p$-adic Pila-Wilkie counting after a finite ray-class parametrization of the contributing split primes. For rational non-integral weights, a local ramification argument gives finiteness. The imaginary quadratic hypothesis gives the required square-root height bound through the positive-definite norm form.

Section~\ref{sec:local} records the local preliminaries, Section~\ref{sec:parametrization} gives the finite parametrization of the contributing split Frobenius primes, and Sections~\ref{sec:rigidity}-\ref{sec:rational} prove the main results. At the end, we give some remarks on the theorem.

\section{Local weights}\label{sec:local}

Throughout, $p$ is an odd prime and $E/\mathbb Q_p$ is finite.  We first recall some basic facts about number fields and $p$-adic Hodge theory.

\begin{lemma}\label{lem:finite-ramification}
Let $K$ be a number field and let
\[
\chi:G_K\longrightarrow\mathcal O_E^\times
\]
be a continuous character.  Then $\chi$ is unramified outside a finite set of primes of $K$.

Consequently,
\[
\operatorname{Ind}_{G_K}^{G_{\mathbb Q}}\chi
\]
is unramified outside finitely many rational primes.
\end{lemma}

We also use the standard local fact that, for $\ell\ne p$, a continuous character
\[
L^\times\longrightarrow\mathcal O_E^\times,
\qquad L/\mathbb Q_\ell\text{ finite},
\]
has finite image on $\mathcal O_L^\times$.  This follows immediately from the decomposition of $\mathcal O_L^\times$ into a finite group and a pro-$\ell$ group.

Let $L$ be a finite-dimensional commutative \'etale $\mathbb Q_p$-algebra and let $\mathcal O_L$ be its maximal compact subring.  By the $p$-adic logarithm, every continuous character
\[
\lambda:L^\times\longrightarrow\mathcal O_E^\times
\]
has, on a sufficiently deep principal-unit subgroup, a unique expression
\begin{equation}\label{eq:analytic-character}
\lambda(u)=\exp\bigl(-\kappa(\log u)\bigr),
\end{equation}
where $\kappa\in\operatorname{Hom}_{\mathbb Q_p}(L,E)$.  This is the usual weight formalism for continuous $p$-adic characters.  In the notation of Kedlaya-Pottharst-Xiao \cite[Definition~6.1.21]{KedlayaPottharstXiao2014}, applied factorwise when $L$ is \'etale, $\kappa$ is the weight of $\lambda^{-1}$, or equivalently the weight of $\lambda$ is $-\kappa$.  Lemma~6.1.23 of \cite{KedlayaPottharstXiao2014} identifies these character weights with the corresponding Sen weights for rank-one $(\varphi,\Gamma)$-modules.  We always choose the depth so that the logarithm and exponential occurring here are defined and injective.

We now specialize to the quadratic case.  We normalize local Artin reciprocity so that a uniformizer maps to arithmetic Frobenius, and use the compatible global Artin map for which principal id\`eles map to the identity.

Let $K/\mathbb Q$ be quadratic, let $c$ be its non-trivial automorphism, and put $K_p=K\otimes_{\mathbb Q}\mathbb Q_p$.  For each $\mathfrak p\mid p$, let
\[
\lambda_{\mathfrak p}:K_{\mathfrak p}^\times\longrightarrow\mathcal O_E^\times
\]
be the local character attached to $\chi|_{G_{K_{\mathfrak p}}}$, and let
\[
\lambda_p:K_p^\times\longrightarrow\mathcal O_E^\times
\]
be their product.  Write $\kappa\in\operatorname{Hom}_{\mathbb Q_p}(K_p,E)$ for the functional in \eqref{eq:analytic-character}.  After replacing $E$ by a finite extension if necessary, fix an embedding
\[
\tau:K\hookrightarrow E\subset\overline{\mathbb Q}_p
\]
and write $\tau^c=\tau\circ c$ for the other embedding.  There are unique $a,b\in E$ such that
\begin{equation}\label{eq:weight-decomposition}
\kappa=a\tau+b\tau^c.
\end{equation}
Equivalently, in the notation of Kedlaya-Pottharst-Xiao,
\[
a=\operatorname{wt}_{\tau}(\lambda_p^{-1}),
\qquad
b=\operatorname{wt}_{\tau^c}(\lambda_p^{-1}).
\]
We refer simply to $a,b$ as the two weights of $\chi$ at $p$.

Reordering the two embeddings swaps $a$ and $b$, and further extension of the coefficient field does not change them.  Thus the conditions
\[
(a,b)\in\mathbb Q^2,
\qquad
(a,b)\in\mathbb Z^2,
\qquad
(a,b)\notin\mathbb Q^2
\]
are intrinsic.

We also use the standard local algebraicity criterion for one-dimensional $p$-adic representations.  If $\delta:G_L\to\mathcal O_F^\times$ is continuous and, after passage to a splitting field, its local character has the form
\[
\lambda_\delta(u)
=
\exp\left(-\sum_{\tau:L\hookrightarrow F^{\mathrm{spl}}}
 k_\tau\,\tau(\log u)\right)
\]
on sufficiently deep principal units, then $\delta$ is locally algebraic if and only if every $k_\tau$ is an integer; equivalently, the underlying $\mathbb Q_p$-representation is Hodge-Tate, or de Rham.  We use Tate's theorem together with Conrad's formulation of locally algebraic characters, including the descent of the integral exponent vector; see \cite[Definition~B.1 and Lemma~B.2]{Conrad2011} and \cite[Lemma~4.1]{Conrad2011}.  De Rhamness is preserved and detected after finite extension of $p$-adic fields \cite[Proposition~6.3.8]{BrinonConrad2009}.  To pass from $\chi$ to the induced representation, restrict locally to the decomposition group or groups above $p$.  After the corresponding finite extension, the restriction of
\[
\rho=\operatorname{Ind}_{G_K}^{G_{\mathbb Q}}\chi
\]
is the direct sum of the local character attached to $\chi$ and its conjugate.  Since de Rham representations are stable under finite direct sums and direct summands, $\rho$ is de Rham at $p$ if and only if the two local rank-one characters are de Rham.  The preceding local algebraicity criterion therefore gives
\begin{equation}\label{eq:derham-integral}
\rho\text{ is de Rham at }p
\quad\Longleftrightarrow\quad
a,b\in\mathbb Z.
\end{equation}

\section{Finite parametrization of split Frobenius primes}\label{sec:parametrization}

From now on, $K/\mathbb Q$ is imaginary quadratic, $c$ denotes complex conjugation, and
\[
\rho=\operatorname{Ind}_{G_K}^{G_{\mathbb Q}}\chi
\]
with $\chi$ a continuous $p$-adic character.  Its ramification set is finite by \cref{lem:finite-ramification}.  No splitting condition on $p$ in $K$ is imposed.

We now introduce the ray-class quotient used to parametrize the split Frobenius primes; we use the standard ideal and id\`ele formulations of class field theory as in \cite[Chapters~VI-VII]{Neukirch1999}.

Let $\Psi$ be the multiplicative character on fractional ideals prime to the ramification set of $\chi$, defined at an unramified prime ideal $\mathfrak q$ by
\[
\Psi(\mathfrak q)=\chi(\operatorname{Frob}_{\mathfrak q}).
\]
Using the standard local-unit fact recalled in Section~\ref{sec:local}, choose an integral modulus $\mathfrak m$, prime to $p$, such that
\begin{enumerate}[(a)]
\item $\overline{\mathfrak m}=\mathfrak m$;
\item every prime away from $p$ at which $\chi$ ramifies divides $\mathfrak m$;
\item for every $v\mid\mathfrak m$, the local character is trivial on $1+\mathfrak m\mathcal O_{K_v}$.
\end{enumerate}
The conjugation stability in (a) is needed when the conjugate prime is treated.

Let $I_{\mathfrak m}^{(p)}$ be the group of fractional ideals prime to $p\mathfrak m$, and let $P_{\mathfrak m}^{(p)}$ be the subgroup of principal ideals $(\alpha)$ satisfying
\[
\alpha\equiv1\pmod{\mathfrak m},
\qquad
v_{\mathfrak p}(\alpha)=0\quad(\mathfrak p\mid p).
\]
If $\alpha\in K^\times$ satisfies $\alpha\equiv1\pmod{\mathfrak m},
v_{\mathfrak p}(\alpha)=0(\mathfrak p\mid p)$, then $(\alpha)$ is prime to $p\mathfrak m$.

The quotient
\[
I_{\mathfrak m}^{(p)}/P_{\mathfrak m}^{(p)}
\]
is finite, since it injects into the ordinary ray class group modulo $\mathfrak m$.

Choose a finite list of fractional-ideal representatives
\[
\mathfrak a_1,\ldots,\mathfrak a_J
\]
prime to $p\mathfrak m$, containing a representative of every class and stable under conjugation.

For an element $\alpha\in K^\times$ satisfying the two conditions
above, global reciprocity gives the following relation between the
ideal character and the local character at $p$. 
\[
1=\prod_v\chi(\operatorname{Art}_v(\alpha))
  =\Psi((\alpha))\lambda_p(\alpha).
\]
Indeed, the factors at $v\mid\mathfrak m$ are trivial, the unramified factors away from $p\mathfrak m$ give $\Psi((\alpha))$, and the factors above $p$ give $\lambda_p(\alpha)$.  Hence
\begin{equation}\label{eq:principal-reciprocity}
\Psi((\alpha))=\lambda_p(\alpha)^{-1}.
\end{equation}

We next study how the split primes contributing to the fixed-trace condition can be organized into finitely many arithmetic families.  Since the trace of an index-two induction vanishes at inert primes, the assumption $r\ne0$ allows us, outside a fixed finite exceptional set, to consider only primes
\[
\ell\mathcal O_K=\mathfrak l\overline{\mathfrak l}
\]
that split in $K$.  The class of $\mathfrak l$ in the finite quotient above is represented by some $\mathfrak a_j$, and therefore
\[
\mathfrak l\mathfrak a_j^{-1}=(\alpha)
\]
for an element $\alpha\in K^\times$ satisfying
\[
\alpha\equiv1\pmod{\mathfrak m},
\qquad
v_{\mathfrak p}(\alpha)=0\quad(\mathfrak p\mid p).
\]
We call such an $\alpha$ an \emph{admissible generator}.  For this representative choose a $\mathbb Z$-basis $\omega_1,\omega_2$ of $\mathfrak a_j^{-1}$, so that
\[
\alpha=m\omega_1+n\omega_2,
\qquad m,n\in\mathbb Z.
\]

We now refine this finite ray-class decomposition at $p$.  Choose $h$ sufficiently large that the logarithm and all analytic powers occurring below are defined on
\[
1+p^h(\mathcal O_K\otimes\mathbb Z_p).
\]
The quotient
\[
(\mathcal O_K\otimes\mathbb Z_p)^\times/
\bigl(1+p^h(\mathcal O_K\otimes\mathbb Z_p)\bigr)
\]
is finite.  On each nonempty residue class of admissible generators, choose one element $\alpha_0$.  If
\[
\alpha_0=m_0\omega_1+n_0\omega_2,
\]
then every generator in the same class satisfies
\[
(m,n)\in(m_0,n_0)+p^h\mathbb Z^2,
\]
and
\[
z=\alpha/\alpha_0\in1+p^h(\mathcal O_K\otimes\mathbb Z_p).
\]
Indeed,
$\mathfrak a_j^{-1}\otimes\mathbb Z_p
=\mathbb Z_p\omega_1\oplus\mathbb Z_p\omega_2
=\mathcal O_K\otimes\mathbb Z_p$,
and $\alpha_0$ is a $p$-adic unit; hence the coordinate differences lie in $p^h\mathbb Z_p$, and being integers they lie in $p^h\mathbb Z$.

Fix an embedding $\iota_p:\overline{\mathbb Q}\hookrightarrow\overline{\mathbb Q}_p$ extending $\tau|_K$, and put
\[
x=\tau(z),
\qquad
y=\tau^c(z).
\]
Then $x$ and $y$ lie in sufficiently deep principal-unit discs.  By \eqref{eq:principal-reciprocity} and the weight decomposition,
\[
\Psi(\mathfrak l)
=\Psi(\mathfrak a_j)\lambda_p(\alpha_0)^{-1}x^ay^b.
\]
Applying the same formula to the conjugate ideal, and using $\tau\circ c=\tau^c$ and $\tau^c\circ c=\tau$, gives
\[
\Psi(\overline{\mathfrak l})
=\Psi(\overline{\mathfrak a_j})\lambda_p(c(\alpha_0))^{-1}x^by^a.
\]
Thus, with
\begin{equation}\label{eq:AB-param-explicit}
A=\Psi(\mathfrak a_j)\lambda_p(\alpha_0)^{-1}, B=\Psi(\overline{\mathfrak a_j})\lambda_p(c(\alpha_0))^{-1},
\end{equation}
the formula of the trace at split prime ideals becomes
\begin{equation}\label{eq:trace-normal-form}
\operatorname{tr}\rho(\operatorname{Frob}_\ell)
=Ax^ay^b+Bx^by^a.
\end{equation}
 Taking norms in $\mathfrak l\mathfrak a_j^{-1}=(\alpha)$ and using $xy=N_{K/\mathbb Q}(z)$ gives
\begin{equation}\label{eq:norm-relation}
xy=\frac{\ell}{N(\mathfrak a_j)N_{K/\mathbb Q}(\alpha_0)}=\gamma\ell,
\qquad
\gamma:=\frac{1}{N(\mathfrak a_j)N_{K/\mathbb Q}(\alpha_0)}\in\mathbb Q_{>0}.
\end{equation}
Since the norm form of an imaginary quadratic field is positive definite, the fixed lattice basis gives
\begin{equation}\label{eq:height-bound}
|m|+|n|\ll\sqrt\ell.
\end{equation}
Finally, the two embeddings of $K$ give independent linear forms in $(m,n)$, so after scalar extension the map $(m,n)\mapsto(x,y)$ is invertible.

There are only finitely many choices of a ray-class representative $\mathfrak a_j$ and a nonempty admissible $p$-adic residue class of generators.  We call each such pair a \emph{parametrization family}.  

Throughout the remainder of the paper, a fixed finite exceptional set may be enlarged without further notice.  It contains $p$, the rational primes below the ramification sets of $K$ and $\chi$, the primes below the modulus $\mathfrak m$, and the primes below the supports of the finitely many fixed ideal representatives and normalization constants.  Such enlargements do not affect any asserted asymptotic estimate or finiteness statement.

For later use, the preceding construction may be summarized as follows.

\begin{proposition}\label{prop:split-parametrization}
Outside a fixed finite set, every unramified prime $\ell$ with
\[
\operatorname{tr}\rho(\operatorname{Frob}_\ell)=r\ne0
\]
belongs to one of the finitely many parametrization families above. Whenever a family is fixed, all auxiliary choices and constants attached to it are understood to be fixed as well.  If $\alpha=m\omega_1+n\omega_2$ is the associated generator and
\[
z=\alpha/\alpha_0,\qquad x=\tau(z),\qquad y=\tau^c(z),
\]
then
\begin{equation}\label{eq:parametrization-summary}
\operatorname{tr}\rho(\operatorname{Frob}_\ell)
=Ax^ay^b+Bx^by^a,
\qquad
xy=\gamma\ell,
\qquad
|m|+|n|\ll\sqrt\ell,
\end{equation}
where the implied constant depends only on the fixed family.  Moreover
\[
(m,n)\in(m_0,n_0)+p^h\mathbb Z^2.
\]
\end{proposition}

In the point-counting argument we use the lattice coordinates $(m,n)\in\mathbb Z^2$; by \eqref{eq:height-bound}, primes $\ell\le X$ give height $O(\sqrt X)$ within each fixed family.

We next extend each integral congruence class to a full $p$-adic bidisc.  In other words, we regard the arithmetic coordinates $x=\tau(z)$ and $y=\tau^c(z)$ as analytic functions of $p$-adic parameters $(u,v)$.

\begin{lemma}\label{lem:full-bidisc}
Fix a parametrization family.  The congruence class
\[
(m,n)\in(m_0,n_0)+p^h\mathbb Z^2
\]
is parametrized by $\mathbb Z^2\subset\mathbb Z_p^2$ through
\[
m=m_0+p^hu,
\qquad
n=n_0+p^hv.
\]
For $(u,v)\in\mathbb Z_p^2$, define
\[
\alpha(u,v)
=(m_0+p^hu)\omega_1+(n_0+p^hv)\omega_2
\in\mathfrak a_j^{-1}\otimes\mathbb Z_p=\mathcal O_K\otimes\mathbb Z_p
\]
and
\[
z(u,v)=\frac{\alpha(u,v)}{\alpha_0}.
\]
Then:
\begin{enumerate}[(i)]
\item $z(u,v)\in1+p^h(\mathcal O_K\otimes\mathbb Z_p)$ for every $(u,v)\in\mathbb Z_p^2$;
\item the functions
\[
x(u,v)=\tau(z(u,v)),
\qquad
y(u,v)=\tau^c(z(u,v))
\]
are affine-linear restricted functions.  Writing $\Xi$ for the coefficient matrix of their linear parts, one has
\begin{equation}\label{eq:xy-affine-bidisc}
\begin{pmatrix}
x(u,v)-1\\[2pt]
y(u,v)-1
\end{pmatrix}
=p^h\Xi
\begin{pmatrix}u\\ v\end{pmatrix},
\qquad
\Xi\in M_2(\mathcal O_E),\quad \det\Xi\ne0;
\end{equation}
\item the pulled-back trace function
\begin{equation}\label{eq:param-analytic-function}
f(u,v)=Ax(u,v)^ay(u,v)^b+Bx(u,v)^by(u,v)^a-r
\end{equation}
belongs to $E\langle u,v\rangle$;
\item every prime in the fixed family satisfying $\operatorname{tr}\rho(\operatorname{Frob}_\ell)=r$ determines an integral point $(u,v)\in\mathbb Z^2$ at which $f(u,v)=0$.
\end{enumerate}
Thus \eqref{eq:param-analytic-function} analytically interpolates the arithmetic trace equation of the fixed family on the full affinoid $\operatorname{Sp}E\langle u,v\rangle$.
\end{lemma}

\begin{proof}
By construction,
\[
\alpha_0=m_0\omega_1+n_0\omega_2,
\qquad
\mathfrak a_j^{-1}\otimes\mathbb Z_p=\mathcal O_K\otimes\mathbb Z_p,
\]
and $\alpha_0$ is a unit in $\mathcal O_K\otimes\mathbb Z_p$.  Hence
\[
z(u,v)
=1+p^h\left(
 u\frac{\omega_1}{\alpha_0}
 +v\frac{\omega_2}{\alpha_0}
\right),
\]
which proves (i).  Applying $\tau$ and $\tau^c$ gives \eqref{eq:xy-affine-bidisc} with
\[
\Xi=
\begin{pmatrix}
\tau(\omega_1/\alpha_0)&\tau(\omega_2/\alpha_0)\\
\tau^c(\omega_1/\alpha_0)&\tau^c(\omega_2/\alpha_0)
\end{pmatrix}.
\]
Its determinant is nonzero because
\[
K_p\otimes_{\mathbb Q_p}E
\longrightarrow E\times E,
\qquad
z\longmapsto(\tau(z),\tau^c(z))
\]
is an isomorphism of split \'etale $E$-algebras.  This proves (ii) and shows that $(u,v)\mapsto(x(u,v),y(u,v))$ extends to an affine-linear automorphism of the ambient affine plane.

For (iii), the Gauss norms of $x-1$ and $y-1$ are strictly smaller than $1$, so $\log x$ and $\log y$ converge in $E\langle u,v\rangle$.  The depth $h$ was chosen so that $a\log x$, $b\log x$, $a\log y$, and $b\log y$ lie in the convergence radius of the $p$-adic exponential.  Hence $x^a,x^b,y^a,y^b$, and therefore $f$, belong to $E\langle u,v\rangle$.  The same restricted series continue to converge after any finite complete scalar extension.

Finally, a global generator $\alpha$ arising from a Frobenius prime in the fixed family has
\[
m=m_0+p^hu,
\qquad
n=n_0+p^hv
\]
with $u,v\in\mathbb Z$.  At such a point $\alpha(u,v)=\alpha$, and \eqref{eq:trace-normal-form} gives $f(u,v)=0$ whenever the Frobenius trace equals $r$.  This proves (iv).
\end{proof}

\section{Algebraic curves in the trace locus and rigidity}\label{sec:rigidity}

The point-counting argument below separates an affinoid trace locus into algebraic and transcendental parts.  After passing from the lattice coordinates to the $(x,y)$-coordinates of the preceding section, a positive-dimensional algebraic piece gives an algebraic curve whose analytification contains a nonempty admissible open on which the trace equation holds. 

Fix a parametrization family.  Let
\[
\varphi:\mathbb A_E^2\longrightarrow\mathbb A_E^2,
\qquad
(u,v)\longmapsto\bigl(x(u,v),y(u,v)\bigr)
\]
denote the affine-linear coordinate map supplied by \cref{lem:full-bidisc}.
For every complete extension $F/E$, write
\[
\mathbb B_F^2=\operatorname{Sp}F\langle u,v\rangle
\]
for the closed unit bidisc in the $(u,v)$-coordinates.  Let
\[
\mathbb D_{x,y}=\operatorname{Sp}E\left\langle \frac{x-1}{p^h},\frac{y-1}{p^h}\right\rangle,
\]
the closed product bidisc with $E$-points $(1+p^h\mathcal O_E)^2$; write $\mathbb D_{x,y,F}$ for its base change to a complete extension $F/E$.  The depth $h$ is chosen so that all analytic powers below are defined there.  By \eqref{eq:xy-affine-bidisc},
\[
\varphi(\mathbb B_E^2)\subset\mathbb D_{x,y}.
\]
Moreover, since $\det\Xi\ne0$, there is an integer $c_0\ge0$ with $p^{c_0}\Xi^{-1}\in M_2(\mathcal O_E)$, and hence
\[
(1+p^{h+c_0}\mathcal O_E)^2\subset\varphi(\mathbb B_E^2).
\]
The fixed-trace equation in the $(x,y)$-coordinates is the zero locus of the analytic function
\begin{equation}\label{eq:trace-function}
\Phi(x,y)=Ax^ay^b+Bx^by^a-r.
\end{equation}

We use the affinoid identity theorem and one-variable Weierstrass
preparation in the forms given in \cite{BGR1984}, together with
Conrad's compatibility of irreducible components with analytification
\cite[Theorem~2.3.1]{Conrad1999}.

Note that $\Phi \not\equiv 0 $ on $\mathbb D_{x,y}$ if $(a,b)\neq (0,0)$.  Indeed, put $d=a-b$.  If $d\ne0$, restrict
to the analytic subdisc $xy=1$.  With $t=x/y$ and $Z=t^{d/2}$, the function
$Z$ is nonconstant because
\[
\frac{\mathrm dZ}{Z}=\frac d2\frac{\mathrm dt}{t}\ne0,
\]
whereas an identity $\Phi=0$ would give
\[
AZ^2-rZ+B=0.
\]
After factoring this quadratic over a finite coefficient extension, the
integral-domain property of the one-variable Tate algebra would force $Z$ to
be constant, a contradiction.  If $d=0$, then $a=b\ne0$ and
\[
\Phi=(A+B)(xy)^a-r,
\]
which is either the nonzero constant $-r$ or a nonconstant analytic function.

The following rigidity statement is the main geometric input.

\begin{theorem}\label{thm:curve-rigidity}
Assume $r\ne0$ and $(a,b)\ne(0,0)$.  Let $L/E$ be finite, let $C\subset\mathbb A^2_L$ be an irreducible algebraic curve, and suppose that $\Phi$ vanishes on a nonempty admissible open
\[
\Omega\subset C^{\mathrm{an}}\cap\mathbb D_{x,y,L}.
\]
Then at least one of the following holds:
\begin{enumerate}[(i)]
\item $a,b\in\mathbb Q$;
\item $a+b=0$ and the rational function $x/y$ is constant on $C$;
\item $a-b=0$ and the rational function $xy$ is constant on $C$.
\end{enumerate}
\end{theorem}

Here and below, a rational function on an integral curve over $L$ is called \emph{constant} if it is algebraic over $L$ inside the function field; in characteristic zero this is equivalent to the vanishing of its differential over $L$.

\begin{proof}
Let $\widetilde C$ be the smooth projective normalization of the
projective closure of $C$.  After shrinking $\Omega$, we may assume
that it lies in the smooth locus of $C^{\mathrm{an}}$, and hence view
it as a nonempty admissible open of $\widetilde C^{\mathrm{an}}$.
The coordinate functions $x,y$ then define nonzero rational functions
on $\widetilde C$.

Put $q=xy,\qquad t=x/y,\qquad s=a+b,\qquad d=a-b.$

Thus $q,t\in L(\widetilde C)^\times$.  On $\Omega$,
\[
x^ay^b=q^{s/2}t^{d/2},
\qquad
x^by^a=q^{s/2}t^{-d/2},
\]
and the trace equation becomes
\begin{equation}\label{eq:analytic-trace-qt}
q^{s/2}\bigl(At^{d/2}+Bt^{-d/2}\bigr)=r.
\end{equation}

We first dispose of the cases $s=0$ and $d=0$.  Suppose that $s=0$.
Then $d\ne0$, and \eqref{eq:analytic-trace-qt} reduces to
\[
At^{d/2}+Bt^{-d/2}=r.
\]
Logarithmic differentiation gives
\[
d\,\frac{B-At^d}{B+At^d}\frac{\mathrm dt}{t}=0
\]
on $\Omega$.  The denominator is nowhere zero there, since $r\ne0$.
If $B-At^d$ is not identically zero, then $\mathrm dt$ vanishes on a
nonempty admissible subopen and hence vanishes identically on
$\widetilde C$.  If $B-At^d$ is identically zero, differentiating
$t^d=B/A$ gives the same conclusion.  Thus $t$ is constant, and
alternative~(ii) holds.

Suppose next that $d=0$.  Then $s\ne0$, and
\[
(A+B)q^{s/2}=r.
\]
Since $r\ne0$, we have $A+B\ne0$, and logarithmic differentiation
gives $\mathrm dq=0$.  Hence $q$ is constant, giving~(iii).

We may therefore assume that $s,d\ne0$.  If $t$ were constant, then
the logarithmic differential of \eqref{eq:analytic-trace-qt} would
give
\[
s\frac{\mathrm dq}{q}=0,
\]
so $q$ would also be constant.  Since $x^2=qt, y^2=q/t,$ both $x$ and $y$ would then be constant, contradicting the fact that
$C$ is a curve.  Similarly, if $q$ were constant, then
\[
d\,\frac{B-At^d}{B+At^d}\frac{\mathrm dt}{t}=0.
\]
As above, this forces $t$ to be constant, again a contradiction.
Thus both $q$ and $t$ are nonconstant.

Since $\Omega^1_{L(\widetilde C)/L}$ is one-dimensional over
$L(\widetilde C)$ and $\mathrm dt\ne0$, there is
\[
\theta\in L(\widetilde C)
\]
such that
\[
\frac{\mathrm dq}{q}
=
\theta\frac{\mathrm dt}{t}.
\]
Taking the logarithmic differential of
\eqref{eq:analytic-trace-qt} gives
\[
s\frac{\mathrm dq}{q}
=
d\,\frac{B-At^d}{B+At^d}\frac{\mathrm dt}{t},
\]
and hence
\[
s\theta
=
d\,\frac{B-At^d}{B+At^d}.
\]
Solving for $t^d$, we obtain
\begin{equation}\label{eq:td-rational}
t^d
=
\frac{B(d-s\theta)}{A(d+s\theta)}
\end{equation}
on a nonempty admissible subopen of $\Omega$.  The denominator cannot
vanish identically, since $d+s\theta=0$ together with the preceding
identity would imply $2B=0$.  Thus
\[
R:=
\frac{B(d-s\theta)}{A(d+s\theta)}
\in L(\widetilde C)^\times
\]
is a nonzero rational function agreeing with $t^d$ on a nonempty
admissible open.  Consequently
\[
\frac{\mathrm dR}{R}
=
d\frac{\mathrm dt}{t}
\]
there, and hence on $\widetilde C$ as an identity of rational
differentials.

Choose a zero or pole $P_t$ of the nonconstant rational function $t$.
Taking residues at $P_t$ gives
\[
\operatorname{ord}_{P_t}(R)
=
d\,\operatorname{ord}_{P_t}(t).
\]
Since $\operatorname{ord}_{P_t}(t)\ne0$, it follows that
\[
d=
\frac{\operatorname{ord}_{P_t}(R)}
     {\operatorname{ord}_{P_t}(t)}
\in\mathbb Q.
\]

Moreover, $AR+B$ is not identically zero.  Otherwise $R$ would be
constant, and the identity
\[
\frac{\mathrm dR}{R}
=
d\frac{\mathrm dt}{t}
\]
would contradict $d\ne0$ and the nonconstancy of $t$.  We may
therefore define
\[
S=
\frac{r^2R}{(AR+B)^2}
\in L(\widetilde C)^\times.
\]
Using \eqref{eq:analytic-trace-qt} and $R=t^d$, we find that
\[
S=q^s
\]
on a nonempty admissible open.  Hence
\[
\frac{\mathrm dS}{S}
=
s\frac{\mathrm dq}{q}
\]
as rational differentials on $\widetilde C$.  Taking residues at a
zero or pole $P_q$ of the nonconstant rational function $q$ gives
\[
s=
\frac{\operatorname{ord}_{P_q}(S)}
     {\operatorname{ord}_{P_q}(q)}
\in\mathbb Q.
\]
Therefore
\[
a=\frac{s+d}{2},
\qquad
b=\frac{s-d}{2}
\]
are rational, which is alternative~(i).
\end{proof}

It remains to treat the constant-product and constant-ratio alternatives.

\begin{proposition}\label{prop:degenerate-finite}
Assume $(a,b)\notin\mathbb Q^2$ and fix a parametrization family.  The rational primes in this family for which the corresponding $(x,y)$-points lie on a level curve
\[
xy=q_0
\qquad\text{or}\qquad
x/y=t_0
\]
arising from the constant-product or constant-ratio alternatives of \cref{thm:curve-rigidity} form a finite set.
\end{proposition}

\begin{proof}
If $a+b=0$, the possible ratio levels are zeros of the fixed
one-variable analytic function
\[
AT^a+BT^{-a}-r.
\]
If $a-b=0$, the possible product levels are zeros of
\[
(A+B)Q^a-r.
\]
Neither function vanishes identically, so Weierstrass preparation
gives only finitely many possible levels.

For fixed $q_0$, the relation $xy=q_0$ together with \eqref{eq:norm-relation} determines at most one rational prime $\ell$.  For fixed $t_0$, the condition $x/y=t_0$ is equivalent, for a global generator $\alpha\in K$, to
\[
\tau(\alpha)
=t_0\frac{\tau(\alpha_0)}{\tau^c(\alpha_0)}\tau^c(\alpha).
\]
Its solution set in $K$ is a $\mathbb Q$-subspace of dimension at most one: if two nonzero solutions $\alpha_1,\alpha_2$ were linearly independent over $\mathbb Q$, then
\[
\tau(\alpha_1/\alpha_2)=\tau^c(\alpha_1/\alpha_2),
\]
so $\alpha_1/\alpha_2$ would be fixed by the non-trivial automorphism of $K/\mathbb Q$ and hence belong to $\mathbb Q$.  Thus the admissible generators lie in the intersection of the fixed fractional ideal with a rational line.  If this intersection is nonzero, after choosing a generator and clearing a fixed rational denominator, every admissible generator has the form
\[
\alpha=k\alpha_{t_0},
\qquad k\in\mathbb Z,
\]
for a fixed nonzero $\alpha_{t_0}$.  The norm relation becomes
\[
A_0k^2=B_0\ell
\]
with fixed nonzero integers $A_0,B_0$.   If a prime $q\nmid A_0B_0$ divided $k$, then
\[
2v_q(k)=v_q(\ell)\in\{0,1\},
\]
which is impossible because the left side is a positive even integer.  Hence every prime divisor of $k$ belongs to the fixed finite set of primes dividing $A_0B_0$.  For such a fixed prime $q$, the equation gives
\[
v_q(A_0)+2v_q(k)
=
v_q(B_0)+\mathbf 1_{q=\ell}.
\]
The right-hand side has only two possible values as $\ell$ varies, so $v_q(k)$ has only finitely many possibilities.  Thus $k$ itself has only finitely many possible values (up to sign), and then $A_0k^2=B_0\ell$ determines $\ell$.  Hence only finitely many $k$ and $\ell$ occur.
\end{proof}

\begin{corollary}\label{cor:diagonal-antidiagonal}
Assume $r\ne0$.  If either
\[
a=b\notin\mathbb Q
\qquad\text{or}\qquad
a=-b\notin\mathbb Q,
\]
then
\[
\pi_{\rho,r}(X)=O_{\rho,r}(1).
\]
\end{corollary}

\begin{proof}
Fix a parametrization family.  If $a=b\notin\mathbb Q$, then the trace equation is
\[
(A+B)(xy)^a-r=0.
\]
This is a nonzero one-variable analytic function of $q=xy$: if $A+B=0$, it is the nonzero constant $-r$, and otherwise it is nonconstant.  It therefore has only finitely many zeros on the fixed closed principal-unit disc.  Each zero $q_0$ determines at most one rational prime through $q_0=\gamma\ell$.

If $a=-b\notin\mathbb Q$, then the trace equation is the nonzero one-variable analytic equation
\[
A(x/y)^a+B(x/y)^{-a}-r=0.
\]
It has only finitely many zeros $t_0$ on the fixed ratio disc.  For each such $t_0$, the fixed-ratio argument in the proof of \cref{prop:degenerate-finite} shows that only finitely many rational primes can occur.  Summing over the finitely many parametrization families proves the result.
\end{proof}

\section{Non-rational weights}\label{sec:nonrational}

For point counting we use the rigid-analytic Pila-Wilkie theorem of Binyamini-Kato \cite{BinyaminiKato2025}, in the framework of \cite{PilaWilkie2006}.  For the earlier non-Archimedean subanalytic Pila-Wilkie theory, see \cite{CluckersComteLoeser2015}.

The trace equation takes values in $E$, while the point-counting theorem below is applied to $\mathbb Q_p$-rational points.  We therefore use the following elementary coefficient decomposition.  Let $d_E=[E:\mathbb Q_p]$ and choose a $\mathbb Q_p$-basis $e_1,\ldots,e_{d_E}$ of $E$.  If
\[
f=\sum_{\nu}c_\nu T^\nu\in E\langle T_1,\ldots,T_n\rangle
\]
and $c_\nu=\sum_i e_i c_{\nu,i}$ with $c_{\nu,i}\in\mathbb Q_p$, then the equivalence of norms on the finite-dimensional $\mathbb Q_p$-vector space $E$ implies
\[
c_{\nu,i}\longrightarrow0
\qquad\text{for every }i.
\]
Hence
\[
f=e_1f_1+\cdots+e_{d_E}f_{d_E},
\qquad
f_i:=\sum_\nu c_{\nu,i}T^\nu\in\mathbb Q_p\langle T_1,\ldots,T_n\rangle.
\]
For every $x\in\mathbb Q_p^n$ in the closed unit polydisc,
\[
f(x)=0\quad\Longleftrightarrow\quad f_1(x)=\cdots=f_{d_E}(x)=0.
\]
Thus the $\mathbb Q_p$-rational zero set of $f$ is the rational point set of the fixed affinoid cut out by the coordinate functions $f_i$.  For a rigid subspace $Y$ of a closed polydisc over $\mathbb Q_p$, we write
\[
Y(\mathbb Q):=Y(\mathbb Q_p)\cap\mathbb Q^n,
\]
where the intersection is taken in the ambient coordinates; heights below are measured in these coordinates, as in \cite[Section~1.2]{BinyaminiKato2025}.

We use the following form of the rigid-analytic Pila-Wilkie theorem.

\begin{theorem}[Binyamini-Kato]\label{thm:BK}
Let $\mathcal B$ be a $\mathbb Q_p$-affinoid algebra presented in a closed unit polydisc and put $Y=\operatorname{Sp}\mathcal B$.  Following \cite[Sections~2.11-2.13]{BinyaminiKato2025}, call an irreducible closed subset of the ambient polydisc \emph{algebraic} if it is an irreducible component of a closed set defined by an ideal obtained by extension from a polynomial ideal in the ambient affine space.  Let $Y^{\mathrm{alg}}$ be the union of all positive-dimensional irreducible closed algebraic subsets contained in $Y$, and put
\[
Y^{\mathrm{tran}}=Y\setminus Y^{\mathrm{alg}}.
\]
For every $\delta>0$ there is a constant $C(\mathcal B,\delta)$ such that
\[
\#\{P\in Y^{\mathrm{tran}}(\mathbb Q):H(P)\le H\}
\le C(\mathcal B,\delta)H^\delta,
\]
where $H(P)$ is the maximum of the usual multiplicative heights of the rational coordinates of $P$; see \cite[Theorem~1.2.1]{BinyaminiKato2025}.
\end{theorem}

We now relate the algebraic part in the Binyamini-Kato theorem to the trace locus.

\begin{lemma}\label{lem:BK-compatibility}
Assume $r\ne0$ and $(a,b)\ne(0,0)$.  Fix a parametrization family.  Write
\[
f=e_1f_1+\cdots+e_{d_E}f_{d_E},
\qquad f_i\in\mathbb Q_p\langle u,v\rangle,
\]
relative to the fixed $\mathbb Q_p$-basis of $E$, and set
\[
Y=\operatorname{Sp}\bigl(\mathbb Q_p\langle u,v\rangle/(f_1,\ldots,f_{d_E})\bigr).
\]
Recall the affine-linear map $\varphi(u,v)=(x(u,v),y(u,v))$ introduced in Section~\ref{sec:rigidity}.  By \eqref{eq:xy-affine-bidisc}, it is an algebraic automorphism, and its restriction induces a rigid-analytic isomorphism
\[
\mathbb B_E^2\xrightarrow{\sim}\varphi(\mathbb B_E^2).
\]

Then:
\begin{enumerate}[(a)]
\item Every Frobenius prime in the fixed family gives a rational point of $Y(\mathbb Q)$.
\item Every positive-dimensional irreducible algebraic subset $C\subset Y$ in the sense of \cref{thm:BK} has dimension one.
\item Let $C$ be such a subset and let $D$ be any irreducible component of the analytic base change $C_E$.  Then there is an irreducible algebraic curve $Z\subset\mathbb A_E^2$ and a nonempty admissible open
\[
\Omega\subset Z^{\mathrm{an}}\cap\mathbb D_{x,y}
\]
such that
\[
\varphi(D)\subset Z^{\mathrm{an}},
\qquad
\Phi|_\Omega=0.
\]
Thus $Z$ and $\Omega$ satisfy the hypotheses of \cref{thm:curve-rigidity}.
\end{enumerate}
\end{lemma}

\begin{proof}
Part (a) follows from \cref{lem:full-bidisc}(iv) and the coefficient decomposition above.

For (b), the non-vanishing of $\Phi$ and
\[
(1+p^{h+c_0}\mathcal O_E)^2
\subset \varphi(\mathbb B_E^2)
\]
show that $f\ne0$. Hence some $f_i$ is nonzero, and therefore
$\dim Y\le1$.

For (c), by the definition of the Binyamini-Kato algebraic part there is a polynomial ideal $J\subset\mathbb Q_p[u,v]$ such that $C$ is an irreducible component of the rigid zero locus
\[
W=V(J)^{\mathrm{an}}\cap\mathbb B_{\mathbb Q_p}^2
\]
inside the closed unit bidisc.  Since $\dim C=1$, the ideal $J$ is nonzero.  After scalar extension to $E$, every irreducible component $D$ of $C_E$ is an irreducible component of $W_E$ by \cite[Theorem~3.4.2]{Conrad1999}.

Put $X=V(J)_E$.  Then
\[
W_E=X^{\mathrm{an}}\cap\mathbb B_E^2
\]
is an admissible open subspace of $X^{\mathrm{an}}$.  By the compatibility of algebraic irreducible components with analytification and admissible opens \cite[Corollary~2.2.9 and Theorem~2.3.1]{Conrad1999}, there is a unique irreducible algebraic component $Z_0$ of $X$ such that
\[
D\subset Z_0^{\mathrm{an}}.
\]
Moreover, $D$ contains a nonempty admissible open $\Omega_0$ of $Z_0^{\mathrm{an}}$.  Equivalently, after shrinking inside $D$ if necessary, one may arrange
\[
\varnothing\ne\Omega_0\subset D\subset W_E,
\qquad
\Omega_0\subset Z_0^{\mathrm{an}}\cap\mathbb B_E^2,
\]
and $\Omega_0$ is admissible open in $Z_0^{\mathrm{an}}$.  Since $D$ has dimension one, the same is true of $Z_0$; in particular, $Z_0$ is an algebraic curve. 

Set
\[
Z=\varphi(Z_0),
\qquad
\Omega=\varphi(\Omega_0).
\]
Since $\varphi$ is an algebraic automorphism, $Z$ is an irreducible algebraic curve, $\varphi(D)\subset Z^{\mathrm{an}}$, and $\Omega$ is a nonempty admissible open contained in $Z^{\mathrm{an}}\cap\varphi(\mathbb B_E^2)\subset Z^{\mathrm{an}}\cap\mathbb D_{x,y}$.  Finally, every $f_i$ vanishes on $D$, so $f=e_1f_1+\cdots+e_{d_E}f_{d_E}$ vanishes on $\Omega_0$.  Therefore $\Phi|_\Omega=0$. 
\end{proof}

\begin{proposition}\label{prop:nonrational}
Under the hypotheses of \cref{thm:main}, if $(a,b)\notin\mathbb Q^2$, then, for every $\varepsilon>0$,
\[
\pi_{\rho,r}(X)\ll_{\rho,r,\varepsilon}X^\varepsilon.
\]
\end{proposition}

\begin{proof}
Fix a parametrization family.  By \cref{lem:full-bidisc},
\[
m=m_0+p^hu,
\qquad
n=n_0+p^hv,
\qquad (u,v)\in\mathbb Z_p^2,
\]
and the trace equation is $f(u,v)=0$.  Writing $f=e_1f_1+\cdots+e_{d_E}f_{d_E}$, the coordinate equations $f_i=0$ define the affinoid $Y$ of \cref{lem:BK-compatibility}.  Congruence conditions coming from the ray modulus away from $p$ are discarded at this stage; this can only enlarge the set of rational points to be counted.

Let $C$ be a positive-dimensional irreducible algebraic subset of $Y$, and let $P=(u,v)\in C(\mathbb Q)$ arise from a Frobenius prime in the fixed family.  After base change to $E$, the point $P_E$ lies on some irreducible component $D$ of $C_E$.  By \cref{lem:BK-compatibility}, $\varphi(D)$ is contained in an algebraic curve $Z$ with a nonempty admissible open in the trace locus.  Since $(a,b)\notin\mathbb Q^2$, \cref{thm:curve-rigidity} forces either $xy$ or $x/y$ to be constant on $Z$.  Because $\varphi(P_E)\in\varphi(D)\subset Z^{\mathrm{an}}$ is an $E$-rational point and both coordinate functions are regular and nonzero on the bidisc, the corresponding constant value is
\[
q_0=xy(\varphi(P_E))\in E^\times
\qquad\text{or}\qquad
 t_0=(x/y)(\varphi(P_E))\in E^\times.
\]
Thus the original Frobenius point lies on the associated level curve, and \cref{prop:degenerate-finite} shows that such points are supported on finitely many rational primes.  Hence every remaining rational point arising from a Frobenius prime lies in $Y^{\mathrm{tran}}(\mathbb Q)$.

For a remaining prime $\ell\le X$, the parametrization gives a point
$(u,v)\in Y^{\mathrm{tran}}(\mathbb Q)$.

Since\[
u=\frac{m-m_0}{p^h},
\qquad
v=\frac{n-n_0}{p^h},
\]
all family data are fixed and \eqref{eq:height-bound} gives
\[
H(u,v)\ll1+|m|+|n|\ll\sqrt\ell\ll\sqrt X,
\]
where $H$ is the maximum of the usual multiplicative heights of the two rational coordinates.  Apply \cref{thm:BK} with $\delta=2\varepsilon$ to obtain
\[
\#\{P\in Y^{\mathrm{tran}}(\mathbb Q):H(P)\ll\sqrt X\}
\ll X^\varepsilon.
\]
A fixed lattice point determines at most one rational prime through \eqref{eq:norm-relation}; conversely every prime to be counted yields at least one point in one of the finitely many parametrization families.  Summing over those families absorbs the corresponding constants and proves the proposition.
\end{proof}

\section{Rational non-integral weights}\label{sec:rational}

Assume throughout this section that
\[
(a,b)\in\mathbb Q^2\setminus\mathbb Z^2.
\]
Write
\begin{equation}\label{eq:rational-weights}
a=\frac{M_1}{N},
\qquad
b=\frac{M_2}{N},
\end{equation}
where
\[
M_1,M_2,N\in\mathbb Z,
\qquad
N>1,
\qquad
\gcd(M_1,M_2,N)=1.
\]
The integers $M_1,M_2$ may be negative.  The denominator $N$ may be divisible by $p$; the analytic powers are defined on a sufficiently deep principal-unit disc, and the finitely many primes $\ell$ dividing $N$ are included in the exceptional set.

Fix a parametrization family.  To separate the two summands in the trace formula, for a Frobenius prime in this family set
\begin{equation}\label{eq:UV}
U=Ax^{M_1/N}y^{M_2/N},
\qquad
V=Bx^{M_2/N}y^{M_1/N}.
\end{equation}
Then, in $\overline{\mathbb Q}_p$,
\begin{equation}\label{eq:UV-sum}
U+V=r.
\end{equation}

We first show that these quantities are algebraic.

\begin{lemma}\label{lem:algebraicity}
Under the fixed embedding $\iota_p:\overline{\mathbb Q}\hookrightarrow\overline{\mathbb Q}_p$, the constants $A,B$ and, for every Frobenius prime in the fixed family, the quantities $U,V$ are images of algebraic numbers.  After replacing them by their unique algebraic representatives, and writing $x=z=\alpha/\alpha_0$ and $y=c(z)$ for the corresponding algebraic representatives of the coordinates, one has in $\overline{\mathbb Q}$
\begin{equation}\label{eq:UV-algebraic-sum}
U+V=r,
\end{equation}
and
\begin{equation}\label{eq:UV-algebraic-powers}
U^N=A^Nx^{M_1}y^{M_2},
\qquad
V^N=B^Nx^{M_2}y^{M_1}.
\end{equation}
\end{lemma}

\begin{proof}
If $u\in K^\times$ is a $p$-unit, then $\lambda_p(u)$ is algebraic when the two weights are rational.  Indeed, for some $k>0$ the element $u^k$ lies in the deep principal-unit subgroup on which the analytic formula is valid, and then
\[
\lambda_p(u)^k
=\lambda_p(u^k)
=\tau(u^k)^{-a}(\tau^c(u^k))^{-b}.
\]
Using \eqref{eq:rational-weights}, the $N$-th power of the right-hand side is
\[
\tau(u^k)^{-M_1}(\tau^c(u^k))^{-M_2}\in\overline{\mathbb Q}.
\]
Thus $\lambda_p(u)^k$ is algebraic, and then $\lambda_p(u)$ is algebraic as a root of a polynomial $T^k-\lambda_p(u)^k$ with algebraic coefficients.

The coordinates $x,y$ are the $p$-adic images of
\[
z=\alpha/\alpha_0,
\qquad
c(z)\in K^\times.
\]
Since
\[
\bigl(x^{M_1/N}y^{M_2/N}\bigr)^N=x^{M_1}y^{M_2},
\]
the chosen analytic root is algebraic over $\iota_p(\overline{\mathbb Q})$ and therefore belongs to this algebraically closed subfield of $\overline{\mathbb Q}_p$.  The same holds with $M_1,M_2$ interchanged.

We next treat the constants.  The class of $\mathfrak a_j$ in $I_{\mathfrak m}^{(p)}/P_{\mathfrak m}^{(p)}$ has finite order, so there exist $k_j>0$ and $\beta\in K^\times$ with
\[
\mathfrak a_j^{k_j}=(\beta),
\qquad
\beta\equiv1\pmod{\mathfrak m},
\qquad
v_{\mathfrak p}(\beta)=0\quad(\mathfrak p\mid p).
\]
Then
\[
\Psi(\mathfrak a_j)^{k_j}=\lambda_p(\beta)^{-1}
\]
by \eqref{eq:principal-reciprocity}, so $\Psi(\mathfrak a_j)$ is algebraic; the same argument applies to $\Psi(\overline{\mathfrak a_j})$.

By construction $\alpha_0$ and $c(\alpha_0)$ are $p$-units, so $\lambda_p(\alpha_0)$ and $\lambda_p(c(\alpha_0))$ are algebraic.  The exact formulas \eqref{eq:AB-param-explicit} therefore give
\[
A,B\in\iota_p(\overline{\mathbb Q})^\times.
\]
The same is then true of $U$ and $V$.  Replacing $A,B,U,V$ by their unique preimages under $\iota_p$, and writing $x=z$ and $y=c(z)$, injectivity of $\iota_p$ turns \eqref{eq:UV-sum} and the defining $N$-th power identities into \eqref{eq:UV-algebraic-sum} and \eqref{eq:UV-algebraic-powers}.
\end{proof}

From now on we use the same symbols $A,B,x,y,U,V$ for these algebraic quantities and for their images under any chosen local embedding.

\begin{lemma}\label{lem:unit-root-unramified}
Let $F$ be a non-Archimedean local field of residue characteristic $\ell$, let $N\ge1$ satisfy $\ell\nmid N$, and let $c\in\mathcal O_F^\times$.  If $w\in\overline F$ satisfies
\[
w^N=c,
\]
then $F(w)/F$ is unramified.
\end{lemma}

\begin{proof}
The $\mathcal O_F$-algebra
\[
\mathcal A=\mathcal O_F[T]/(T^N-c)
\]
is finite and free.  In $\mathcal A$, the class of $T$ is a unit because $T^N=c\in\mathcal O_F^\times$, and $N$ is a unit because $\ell\nmid N$.  Hence the derivative
\[
NT^{N-1}
\]
is a unit in $\mathcal A$, so $\mathcal A$ is finite \'etale over $\mathcal O_F$.  Since $\mathcal O_F$ is henselian, every field factor of $\mathcal A\otimes_{\mathcal O_F}F$ is an unramified extension of $F$.  Evaluation at $w$ selects one of these field factors, which is isomorphic to $F(w)$.
\end{proof}

\begin{proposition}\label{prop:rational-finite}
Under the hypotheses of \cref{thm:main}, if $(a,b)\in\mathbb Q^2\setminus\mathbb Z^2$, then only finitely many rational primes $\ell$ satisfy
\[
\operatorname{tr}\rho(\operatorname{Frob}_\ell)=r.
\]
\end{proposition}

\begin{proof}
Fix a parametrization family.  Suppose first that $M_1=M_2=M$.  Since $N>1$ and $\gcd(M_1,M_2,N)=1$, one has $M\ne0$.  The logarithm is additive on the fixed deep disc, so
\[
x^{M/N}y^{M/N}=(xy)^{M/N}.
\]
Hence the trace equation and \eqref{eq:norm-relation} give
\[
(A+B)(xy)^{M/N}=r,
\qquad
xy=\gamma\ell.
\]
The first identity implies $A+B\ne0$.  Raising to the $N$-th power gives
\begin{equation}\label{eq:equal-weight-norm}
r^N=(A+B)^N(\gamma\ell)^M.
\end{equation}
All factors other than $\ell$ are fixed within the family.  Since $M\ne0$, the map $\ell\mapsto\ell^M$ is injective on positive rational primes, so \eqref{eq:equal-weight-norm} permits at most one $\ell$ in this family.

Assume now that $M_1\ne M_2$. Choose a number field $L\subset\overline{\mathbb Q}$ containing
$K$, $A$, $B$, and the algebraic data of the family.  After enlarging
a finite set $\mathcal S$ of rational primes, we may assume that for
every $\ell\notin\mathcal S$:
\[
\ell\nmid prN,\qquad
L/\mathbb Q\text{ is unramified at }\ell,
\]
all fixed algebraic constants are units at primes above $\ell$, the
fixed fractional ideals are supported away from $\ell$, and $\gamma$
is an $\ell$-adic unit.

Let $\ell\notin\mathcal S$ be a prime in the fixed family.  It splits as
\[
\ell\mathcal O_K=\mathfrak l\overline{\mathfrak l},
\qquad
(\alpha)=\mathfrak l\mathfrak a_j^{-1}.
\]
Choose a prime $\mathfrak P$ of $L$ above the selected $\mathfrak l$, and fix an embedding
\[
\overline{\mathbb Q}\hookrightarrow\overline{L_{\mathfrak P}}
\]
extending the completion embedding of $L$.  Let $v$ be the induced valuation on this algebraic closure, normalized by $v(L_{\mathfrak P}^\times)=\mathbb Z$.  

Because $L/\mathbb Q$ is unramified at $\ell$, the restriction of $v$ to $K$ is the normalized valuation at $\mathfrak l$.  From $(\alpha)=\mathfrak l\mathfrak a_j^{-1}$ and the exclusion of the fixed prime divisors of $\mathfrak a_j$ and $\alpha_0$, we obtain
\[
v(x)=1,
\qquad
v(y)=0.
\]
The fixed algebraic constants are units at $\mathfrak P$, so $v(A)=v(B)=0.$
Using \eqref{eq:UV-algebraic-powers},
\begin{equation}\label{eq:UV-valuations}
v(U)=\frac{M_1}{N},
\qquad
v(V)=\frac{M_2}{N}.
\end{equation}
Since $U+V=r$ and $\ell\nmid r$,
\[
v(U+V)=v(r)=0.
\]
The valuations in \eqref{eq:UV-valuations} are distinct, so
\[
0=\min\left\{\frac{M_1}{N},\frac{M_2}{N}\right\}.
\]
After interchanging $U,V$, we may assume
\[
M_2=0<M_1.
\]
Then $\gcd(M_1,N)=1$.

Let $e=e\bigl(L_{\mathfrak P}(U)/L_{\mathfrak P}\bigr)$. The value group of $L_{\mathfrak P}(U)$ is $\frac1e\mathbb Z$ and contains the reduced fraction $M_1/N$, so $N\mid e$.  On the other hand,
\[
U^N=A^Nx^{M_1}\in L_{\mathfrak P},
\]
whence $[L_{\mathfrak P}(U):L_{\mathfrak P}]\le N$.  Therefore
\[
N\le e\le[L_{\mathfrak P}(U):L_{\mathfrak P}]\le N,
\]
so $L_{\mathfrak P}(U)/L_{\mathfrak P}$ is totally ramified of degree $N>1$.

Since $M_2=0$,
\[
V^N=B^Ny^{M_1}\in\mathcal O_{L_{\mathfrak P}}^\times.
\]
As $\ell\nmid N$, \cref{lem:unit-root-unramified} shows that $L_{\mathfrak P}(V)/L_{\mathfrak P}$ is unramified.  But $V=r-U$, hence
\[
L_{\mathfrak P}(U)=L_{\mathfrak P}(V),
\]
a contradiction.

Thus no primes outside the fixed exceptional set occur when $M_1\ne M_2$, while the equal-weight case contributes at most one prime in each family.  Since there are only finitely many parametrization families, the proposition follows.
\end{proof}

\begin{proof}[Proof of \cref{thm:main}]
Part (i) is \cref{prop:nonrational}, and part (ii) is \cref{prop:rational-finite}.
\end{proof}

\begin{proof}[Proof of \cref{cor:nongeometric}]
The de Rham criterion is \eqref{eq:derham-integral}.  If $\rho$ is not de Rham, then either $(a,b)\notin\mathbb Q^2$ or $(a,b)\in\mathbb Q^2\setminus\mathbb Z^2$, so \cref{thm:main} gives the asserted subpolynomial bound.  Taking any $0<\varepsilon<1/2$ gives
\[
X^\varepsilon=o\!\left(\frac{\sqrt X}{\log X}\right),
\]
and hence the final assertion.
\end{proof}

\begin{remark}[The range of the fixed trace]\label{rem:r-general}
The integrality of $r$ is not essential.  The same proofs give the conclusions of \cref{thm:main} for every fixed $r\in E^\times$.  In the non-rational case, only the condition $r\ne0$ is used.  In the rational-weight case, \cref{lem:algebraicity} shows that every split Frobenius trace lies in
\[
E\cap\iota_p(\overline{\mathbb Q}).
\]
Thus, if $r$ does not lie in this subfield, no prime contributes: inert primes have trace zero and split traces are algebraic.  If $r$ does lie in this subfield, one adjoins its algebraic preimage to the number field used in the proof of \cref{prop:rational-finite} and enlarges the finite exceptional set by the primes at which it is not a unit.  We retain $r\in\mathbb Z\setminus\{0\}$ in the statement in order to match the classical Lang-Trotter formulation.
\end{remark}

\begin{remark}[No irreducibility assumption]\label{rem:no-irreducibility}
No irreducibility hypothesis on $\rho$ is imposed.  In particular, the argument also applies when $\chi$ extends to $G_{\mathbb Q}$ and the induced representation is reducible.
\end{remark}

\begin{remark}\label{rem:further-scope}
The imaginary quadratic hypothesis is used essentially in
\eqref{eq:height-bound}: the norm form is positive definite and gives
$|m|+|n|\ll\sqrt\ell$.  In a real quadratic field, the infinite unit group
prevents this height estimate from following from the norm relation alone.
The present proof also assumes that $p$ is odd, which avoids additional
convergence and torsion bookkeeping for the logarithm, exponential, and the
half-exponents used in Section~\ref{sec:rigidity}.
The integral-weight case $(a,b)\in\mathbb Z^2$ is deliberately excluded
from \cref{thm:main}.  By \eqref{eq:derham-integral} it is the de Rham,
locally algebraic range, where classical CM Lang-Trotter behavior may occur
and neither finiteness nor an $X^\varepsilon$ bound is expected in general.
\end{remark}

\begin{remark}
    In the non-rational range, we expect the stronger general bound
\[
\pi_{\rho,r}(X)=O(1).
\]

\end{remark}

\bibliographystyle{plain}
\bibliography{lang_trotter_quadratic_inductions_copy_ready}

\end{document}